\documentclass{amsart}
\usepackage{amscd,amsmath,latexsym,amsthm,amsfonts,amssymb,graphicx,color,geometry,hyperref,bm}

\usepackage[norefs,nocites]{refcheck}

\hypersetup{
colorlinks=true,
linkcolor=blue,
citecolor=cyan}

\allowdisplaybreaks

\newtheorem{theorem}{Theorem}[section]
\newtheorem{proposition}[theorem]{Proposition}
\newtheorem{corollary}[theorem]{Corollary}

\newtheorem{remark}[theorem]{Remark}

\newtheorem{lemma}[theorem]{Lemma}
\newtheorem{problem}[theorem]{Problem}

\newtheorem{conjecture}[theorem]{Conjecture}

\numberwithin{equation}{section}

\begin{document}
\title[Unimodular multilinear forms on sequence spaces with small norms]{Asymptotic estimates for unimodular multilinear forms with small norms on sequence spaces}

\author[Albuquerque]{Nacib Gurgel Albuquerque}
\address{Departamento de Matem\'{a}tica, \newline \indent
Universidade Federal da Para\'{i}ba,  \newline \indent
58.051-900 - Jo\~{a}o Pessoa, Brazil.}
\email{ngalbqrq@gmail and ngalbuquerque@mat.ufpb.br}

\author[Rezende]{Lisiane Rezende}
\address{Departamento de Matem\'{a}tica, \newline\indent
Universidade Federal da Para\'{i}ba, \newline\indent
58.051-900 - Jo\~{a}o Pessoa, Brazil.}
\email{lirezendestos@gmail.com}

\thanks{N. G. Albuquerque is supported by CNPq 409938/2016-5 and Grant 2019/0014 Para\'iba State Research Foundation (FAPESQ). L. Rezende is supported by CAPES}

\begin{abstract}
The existence of unimodular forms with small norms on sequence spaces is crucial in a variety of problems in modern analysis. We prove that the infimum of $\left\Vert A\right\Vert $ over all unimodular $d$-linear (complex or real) forms $A$ on $\ell_{p_1}^{n_{1}} \times \cdots \times \ell_{p_d}^{n_{d}}$, for all $p_1,\dots,p_d \in [2, \infty]$ and all positive integers $n_1,\dots,n_d$, behaves (asymptotically) as $\left(n_{1}^{1/2} + \cdots + n_{d}^{1/2}\right) \prod_{j=1}^{d}n_{j}^{\frac{1}{2} - \frac{1}{p_j}}$. Applications to the theory of the multilinear Hardy--Littlewood inequality are also presented.
\end{abstract}

\maketitle

\section{Introduction}

The Kahane--Salem--Zygmund inequality provides multilinear forms with unimodular coefficients and small norms on classical sequence spaces. It is nowadays a fundamental tool in different fields of modern Analysis with a broad range of applications involving, for instance, Bohr's radius, the Bohnenblust--Hille and Hardy--Littlewood multilinear inequalities  (see, e.g., \cite{ABPS_HL,Bayart,BPS_adv,BK}). For $p_1,\dots,p_d\geq2$ and $\mathbb{K}$ the real or complex scalar field, it asserts that there exists a $d$-linear form $A:\ell_{p_1}^{n} \times\cdots\times \ell_{p_d}^{n} \to \mathbb{K}$ of the form 
\begin{equation*}
A\left( x^{1},\ldots,x^{d}\right) =\sum_{j_{1},\ldots,j_{d}=1}^{n} 
\varepsilon_{\mathbf{j}}x_{j_{1}}^{1} \cdots x_{j_{d}}^{d},
\end{equation*}
with $\varepsilon_{\mathbf{j}}\in\{-1,1\},\, \mathbf{j}:=(j_1,\dots,j_d)$, such that 
\begin{equation*}
\|A\| \leq C_{d} \cdot n^{\frac{d+1}{2} - \sum_{j=1}^d\frac{1}{p_j}}, 
\end{equation*}
and $C_d:= (d!)^{1-\max\left\{ \frac{1}{2}, \frac{1}{\max\{p_{1}, \ldots, p_{d}\}} \right\}} \sqrt{32d\log(6d)}$.

On the other hand, it is possible to verify that there is a constant $B_{d}>0$ such that 
\begin{equation*}
\left\Vert A\right\Vert \geq B_{d} \cdot n^{ \frac{d+1}{2} - \sum_{j=1}^d%
\frac{1}{p_j} } 
\end{equation*}
for all unimodular $d$-linear forms $A:\ell_{p_1}^{n} \times\cdots\times \ell_{p_d}^{n} \to \mathbb{K}$ with $p_1,\dots,p_d \geq2$. So we conclude that the exponent is sharp and we can summarize this sharpness by writing 
\begin{equation}
0<B_{d} \leq \inf \frac{\|A\|}{ n^{\frac{d+1}{2} - \sum_{j=1}^d\frac{1}{p_j}} }
\leq C_{d}<\infty,  \label{1s}
\end{equation}
where $p_1,\dots,p_d \geq2$ and the infimum is calculated over all unimodular $d$-linear forms $A:\ell_{p_1}^{n} \times\cdots\times \ell_{p_d}^{n} \to \mathbb{K}$. From (\ref{1s}) it is obvious that 
\begin{align}  \label{inf.}
0<B_{d}\leq\inf\frac{\left\Vert A\right\Vert }{n^{s}}
\leq C_{d}<\infty \Leftrightarrow s =\frac{d+1}{2} - \sum_{j=1}^d\frac{1}{p_j}.
\end{align}

\begin{problem}
\label{prob} What about the asymptotic behavior of $\inf \left\Vert A\right\Vert $, where the infimum is estimated over all unimodular $d$-linear forms $A:\ell_{p_1}^{n_{1}} \times\cdots\times \ell_{p_d}^{n_{d}} \to \mathbb{K}$?
\end{problem}

The main goal of this paper is to present a solution to this problem. The following simple lemmata shows that the answer to Problem \ref{prob} is not so immediate.

\begin{lemma}
If there are $s_{1},\dots,s_{d} \in \mathbb{R}$ such that there exist $C_{d},B_{d} \in (0,\infty)$ such that
\begin{equation*}
B_{d} \leq \inf \frac{\left\| A \right\|} {n_{1}^{s_{1}} \cdots n_{d}^{s_{d}}%
} \leq C_{d} 
\end{equation*}
for all $n_{1},\dots,n_{d}\geq1$, where the infimum is estimated over all unimodular $d$-linear forms $A:\ell_{p_1}^{n_{1}} \times\cdots\times \ell_{p_d}^{n_{d}} \to \mathbb{K}$, then $s_{1},\dots,s_{d}$ are unique.
\end{lemma}

\begin{proof}
To prove that $s_{1}$ is unique it suffices to consider $n_{2}=\cdots=n_{d}=1 $ and so on.
\end{proof}

\begin{lemma}
There are no $s_{1},\dots,s_{d} \in \mathbb{R}$ such that there exist $C_{d},B_{d}\in(0,\infty)$ such that 
\begin{equation*}
B_{d} \leq \inf \frac{\|A\|} {n_{1}^{s_{1}}\cdots n_{d}^{s_{d}}} \leq C_{d} 
\end{equation*}
for all $n_{1},\dots,n_{d}\geq1$, where the infimum is estimated over all
unimodular $d$-linear forms $A:\ell_{p_1}^{n_{1}} \times\cdots\times \ell_{p_d}^{n_{d}} \to \mathbb{K}$.
\end{lemma}

\begin{proof}
Considering $n_{2}=\cdots=n_{d}=1$ we conclude that $s_{1}\geq1-\frac{1}{p_1}$. In fact, every unimodular form is of the type 
\begin{equation*}
A(x^{1},\dots,x^{d}) = \left( \sum_{j=1}^{n_{1}} \pm x_{j}^{1}\right)
x_{1}^{2}\cdots x_{1}^{d} 
\end{equation*}
and 
\begin{equation*}
\|A\|=n_{1}^{1-\frac{1}{p_1}}. 
\end{equation*}
By the previous lemma, we have $s_{1}=1-\frac{1}{p_1}$. Analogously, 
\begin{equation*}
s_{2}=1-\frac{1}{p_2}, \ldots, s_{d}=1-\frac{1}{p_d}. 
\end{equation*}
Now, considering $n_{1}=\cdots=n_{d}=n$, we would have 
\begin{equation*}
B_{d} \leq \inf \frac{\left\Vert A\right\Vert }{n^{d-\sum_{j=1}^d\frac{1}{p_j%
}}} \leq C_{d} 
\end{equation*}
and this contradicts \eqref{inf.}.
\end{proof}

\bigskip

An answer to Problem \ref{prob} is a function $f(n_{1},\dots,n_{d})$ so that there exist $C_{d},B_{d}\in(0,\infty)$ such that 
\begin{equation*}
B_{d}\leq\inf\frac{\left\Vert A\right\Vert }{f(n_{1},\dots,n_{d})}\leq C_{d} 
\end{equation*}
for all $n_{1},\dots,n_{d}\geq1$, where the infimum is estimated over all unimodular $d$-linear forms $A:\ell_{p_1}^{n_{1}} \times\cdots\times \ell_{p_d}^{n_{d}} \to \mathbb{K}$. This function $f$ describes the asymptotic growth of such multilinear forms $A$ with small norms. We shall show that (see Theorem \ref{op_exp}), for $p_1,\dots,p_d \geq 2$, $f$ can be chosen as 
\begin{equation*}
\left( n_{1}^{1/2}+\cdots+n_{d}^{1/2} \right) \prod_{j=1}^{d} n_{j}^{\frac{1%
}{2}-\frac{1}{p_j}}. 
\end{equation*}

This paper is organized as follows: in Section \ref{sec-ksz}, we revisit the Kahane--Salem--Zygmund multilinear inequality. The original proof of the inequality (see \cite{Boas2000} and \cite[Lemma 3.1]{ABPS_HL}) is adapted and, by using interpolation and duality, the norm estimate is improved when dealing with $p \in [1,2]$. In Section \ref{sec-ksz-opt}, it is proved that the Kahane--Salem--Zygmund inequality we obtained (Theorem \ref{KSZ_gen}) is sharp when $p_{1},\dots,p_{d}\in\lbrack2,\infty]$. In Section \ref{sec-app} some applications are presented.

Recall that $\ell_p^n$ stands for $\mathbb{K}^n$ with the $\ell_p$-norm, $p \in [1,\infty]$. As usual, the open ball with center $x$ and radius $r$ is denoted by $B(x,r)$, its closure by $\overline{B}(x,r)$ and the center and radius are omitted when we deal with the open and closed unit balls. For the sake of clarity we fix some useful notation: for $p_1,\dots,p_d \in (0,\infty)$, we define $\mathbf{p}:=(p_1,\dots,p_d), \ \left| \frac{1}{\mathbf{p}}\right| := \frac{1}{p_1} + \dots + \frac{1}{p_d}$. The cardinal of a set $\mathcal{I}$ is denoted by $\text{card}\, \mathcal{I}$. We also use the usual multi-index notation $\mathbf{j}:=(j_1,\dots,j_d) \in \mathbb{N}^d$ and $q^\prime$ denotes the conjugate of $q \in [1,\infty]$, i.e., $\frac1q+\frac1{q^\prime}=1$.

\section{Kahane--Salem--Zygmund inequality revisited} \label{sec-ksz}

The main result in this section is Theorem \ref{KSZ_gen}, a version of the multilinear Kahane--Salem--Zygmund inequality on the space $\ell_{p_1}^{n_1}\times \cdots \times \ell_{p_d}^{n_d}$ and with sup norm refined when dealing with some $p_k$ between $1$ and $2$. Two basic tools are needed: an useful upper bound for the probability of a non-negative function to be greater than or equal to some positive constant; (see \cite[Lemma 20.1]{Bauer}); and a covering argument (see \cite[p. 333]{Boas2000}). The techniques follow along the lines of \cite{Boas2000}, we present the sketch of the proof for the sake of completeness.

\begin{lemma}[Chebyshev--Markov Inequality]
\label{cheby} Let $\left(\Omega,\Sigma,\mu\right)$ be a measure space, $f$ be measurable function on $\Omega$, and $g: \mathbb{R }\to \mathbb{R}$ be a non-negative function such that it is non-decreasing on the range of $f$ and $g(t)>0$, if $t>0$. Then, for every positive real number $\alpha$, 
\begin{equation*}
\mu \left(\left[ f \geq \alpha \right]\right) \leq \frac{1}{g(\alpha)}
\int_\Omega g \circ f \,d\mu.
\end{equation*}
In particular, 
\begin{equation*}
\mu \left(\left[ |f| \geq \alpha \right]\right) \leq \frac{1}{\alpha^p}
\int_\Omega |f|^p \, d\mu
\end{equation*}
holds for every real $p>0$.
\end{lemma}

\begin{lemma}[Covering argument]
\label{lemma-cov} Let $r$ be a positive real number. Then the open unit ball  $B$ of $\ell_p^n$ can be covered by a collection of open $\ell_p^n$-balls of radius $r$, with the number of balls in the collection not exceeding $\left(1+2r^{-1}\right)^{2n}$, and the centers of the balls lying in the closed unit ball $\overline{B}$ of $\ell_p^n$.
\end{lemma}

The version of the Kahane--Salem--Zygmund inequality on $%
\ell_{p_{1}}^{n_{1}} \times \cdots \times\ell_{p_{d}}^{n_{d}}$ is presented
below.

\begin{proposition}
\label{KSZ} Let $d,n_{1}, \ldots, n_{d} \geq1$ be positive integers and $%
p_{1},\ldots,p_{d} \in\left[ 1,\infty\right]$. Then there exist signs $%
\varepsilon_{\mathbf{j}} = \pm1$ and a $d$-linear map $A:
\ell_{p_{1}}^{n_{1}} \times\cdots\times\ell_{p_{d}}^{n_{d}} \to\mathbb{K}$
of the form 
\begin{align*}
A\left( z^{1},\ldots,z^{d}\right) =\sum_{j_{1}=1}^{n_{1}} \ldots\sum
_{j_{d}=1}^{n_{d}}\varepsilon_{\mathbf{j}} z^{1}_{j_{1}}\cdots z^{d}_{j_{d}},
\end{align*}
such that 
\begin{equation*}
\|A\| \leq C_{d} \cdot \prod_{k=1}^{d}{n_{k}}^{\max\left\{\frac12 - \frac{1}{%
p_{k}}, 0\right\} } \cdot\left( \sum_{k=1}^{d}n_{k}\right) ^{\frac{1}{2}}, 
\end{equation*}
with $C_{d} = 8 (d!) ^{1-\max\left\{ \frac{1}{2}, \frac{1}{p} \right\} } 
\sqrt{\log(1+4d)}$ and $p:= \max\{p_{1},\ldots,p_{d}\}$.
\end{proposition}

\begin{proof}[Sketch of the proof]
We present the steps of the proof. The argument consists on a probabilistic estimate combined with the covering lemma.

\bigskip

\noindent\textsc{Step 1}. The first goal is to make an upper estimate on the
probability that the operator sum has large modulus. Fixed a multilinear map
like in the statement, it can be written in the form 
\begin{equation*}
A\left(t,z\right)=\sum_{\mathbf{k}}^{\prime} \left(r_\mathbf{k}(t) \sum_{\mathbf{j}\sim\mathbf{k}} z_\mathbf{j}\right),
\end{equation*}
where the following notation was used: $z=(z^1,\dots,z^d) \in \overline{B}_{p_1}^{n_1} \times \cdots \times \overline{B}_{p_d}^{n_d}$; $z_\mathbf{j}$ is a shorthand for the monomial $z^1_{j_1}\cdots z^d_{j_d}$; $t$ lies in the interval $I := \left[0,1\right]$; a prime symbol `` $^{\prime }$'' on a sum or product means that it is restricted to $d$-tuples of integers that are arranged in non-decreasing order; $\mathbf{j} \sim \mathbf{k}$ indicates that the $d$-tuples $\mathbf{j}$ and $\mathbf{k}$ are permutations of each other; and for each $d$-tuple of positive integers $\mathbf{k}:=(k_1,\dots,k_d)$ in non-decreasing order, we choose a different Rademacher function $r_\mathbf{k}$.

Let $\lambda$ be an arbitrary positive number (whose value we will specify later in terms of $d,\,n_1, \ldots, n_{d}$ and $p$). Invoking the independence and orthogonality properties of of the Rademacher functions, H\"older's inequality and $\ell_q$-norm inclusions, 
\begin{align*}
\int_I e^{\mbox{\footnotesize Re} \left[\lambda A\left(t,z\right)\right]%
}\,dt &\leq \prod_{\mathbf{k}}^{\prime} e^{\frac{1}{2} \left(\lambda \sum_{%
\mathbf{j}\sim\mathbf{k}} \mbox{\footnotesize Re}\, z_\mathbf{j} \right)^2}
\\
&= e^{\frac{1}{2} \lambda^2 \sum_{\mathbf{k}}^{\prime} \left(\sum_{\mathbf{j}%
\sim\mathbf{k}} \mbox{\scriptsize Re }z_\mathbf{j}\right)^2}. \\
&\leq \exp \left[\frac{1}{2} \lambda^2 \left(d!\right)^{2\left(1-\frac{1}{%
m(p)} \right)} \cdot {n_1}^{2\left(\frac{1}{2}-\frac{1}{M(p_1)}\right)}
\cdots {n_d}^{2\left(\frac{1}{2}-\frac{1}{M(p_d)}\right)} \right]
\end{align*}
where $p:=\max\{p_1,\ldots,p_d\}, \, m(t) := \min\{t,2\}$ and $M(t) :=
\max\{t,2\}$, for a real positive $t$.

\bigskip

\noindent\textsc{Step 2}. The second part of the proof uses Lemmas \ref{cheby} and \ref{lemma-cov} and a simple Lipschitz estimate for $A\left(t,z\right)$. Let $R$ be an arbitrary positive real number (whose value will be specified later in terms of $d,n_1,\dots,n_d$ and $p$). Using the previous upper bound and applying Lemma \ref{cheby} with $g(t):= e^t$ and 
\begin{equation*}
A:= \left\{ t \in I\,:\, \mbox{Re}\left[A\left(t,z\right)\right] \geq R
\right\} = \left\{t\in I \,:\, \lambda \mbox{Re}\left[A\left(t,z\right)%
\right] \geq \lambda R \right\}, 
\end{equation*}
we gain 
\begin{align*}
\mu\left(A\right) &\leq \exp \left[- R\lambda + \frac{1}{2} \lambda^2
\left(d!\right)^{2\left(1-\frac{1}{m(p)}\right)} \cdot {n_1}^{2\left(\frac{1%
}{2}-\frac{1}{M(p_1)}\right)}\cdots {n_d}^{2\left(\frac{1}{2}-\frac{1}{M(p_d)%
}\right)} \right].
\end{align*}
By symmetric reasoning, the probability that $\left|A\left(t,z\right)\right|$
exceeds $\sqrt{2}R$ is at most 
\begin{equation*}
4 \cdot \exp \left[- R\lambda + \frac{1}{2} \lambda^2
\left(d!\right)^{2\left(1-\frac{1}{m(p)}\right)} \cdot {n_1}^{2\left(\frac{1%
}{2}-\frac{1}{M(p_1)}\right)} \cdots {n_d}^{2\left(\frac{1}{2}-\frac{1}{%
M(p_d)}\right)} \right].
\end{equation*}
This is a probabilistic estimate for an arbitrary but fixed $z$. The
covering Lemma \ref{lemma-cov} implies that
\begin{equation*}
\left[ \sup_{\mathbf{z} \in \overline{B}_{p_1}^{n_1} \times \cdots \times 
\overline{B}_{p_d}^{n_d}} \left|A\left(t,\mathbf{z}\right)\right| \geq 2 
\sqrt{2}R \right] \subset \bigcup_{w\in\mathcal{W}} \left[%
\left|A\left(t,w\right)\right| \geq \sqrt{2}R \right].
\end{equation*}

Hence, applying the preceding probabilistic estimate to each point, we get
that the probability $P_\mathcal{S}$ of the set 
\begin{equation*}
\mathcal{S}:=\left[ \sup_{\mathbf{z} \in \overline{B}_{p_1}^{n_1} \times
\cdots \times \overline{B}_{p_d}^{n_d}} \left|A\left(t,\mathbf{z}%
\right)\right| \geq 2\sqrt{2}R \right] 
\end{equation*}
is at most 
\begin{equation*}
4\left(1+4d\right)^{2\sum_{k=1}^{d}n_k} \cdot \exp \left[- R\lambda + \frac{1%
}{2} \lambda^2 \left(d!\right)^{2\left(1-\frac{1}{m(p)}\right)} \cdot
\prod_{k=1}^d {n_k}^{2\left(\frac{1}{2}-\frac{1}{M(p_k)}\right)} 
\right]. 
\end{equation*}

Now taking the following values for the parameters $R$ and $\lambda$, 
\begin{gather*}
R := \left(2\left(d!\right)^{2\left(1-\frac{1}{m(p)}\right)} \prod_{k=1}^d {%
n_k}^{2\left(\frac{1}{2}-\frac{1}{M(p_k)}\right)}
\log\left(8\left(1+4d\right)^{2\sum_{k=1}^{d}n_k}\right)\right)^\frac{1}{2},
\\
\lambda := \frac{R}{\left(d!\right)^{2\left(1-\frac{1}{m(p)}\right)} {n_1}%
^{2\left(\frac{1}{2}-\frac{1}{M(p_1)}\right)} \cdots {n_d}^{2\left(\frac{1}{2%
}-\frac{1}{M(p_d)}\right)}},
\end{gather*}
we conclude that, with these choices, the probability $P_{\mathcal{S}}$ that the supremum of $\left|A\left(t,\cdot\right)\right|$ over $\overline{B}_{p_1}^{n_1} \times \cdots \times \overline{B}_{p_d}^{n_d}$ exceeds $2\sqrt{2}R$ is $P_{\mathcal{S}} \leq 1/2$. Therefore, there exists a particular value $t_0$ such that the supremum of $\left|A\left(t_0,\cdot\right)\right|$ over $\overline{B}_{p_1}^{n_1} \times \cdots \times \overline{B}_{p_d}^{n_d}$ is no more than 
\begin{equation*}
2\sqrt{2}R \leq 8\sqrt{\log\left(1+4d\right)}\left(d!\right)^{\left(1-\frac{1%
}{m(p)}\right)} \prod_{k=1}^{d}{n_k}^{ \left(\frac{1}{2} - \frac{1}{M(p_k)}%
\right)} \left({\sum_{k=1}^{d}n_k}\right)^\frac{1}{2}.
\end{equation*}
The values of the Rademacher functions at $t_0$ produce the signs $\varepsilon_{\mathbf{j}}$ stated.
\end{proof}


Borrowing an argument from \cite{Bayart}, the estimate of the norm is improved when dealing with some $p_{k}$ between $1$ and $2$. We will fix the following notation for the constant involved in the results of this section:
\[
C_{d}:=8(d!)^{1-\max\left\{ \frac{1}{2},\frac{1}{ \max\{p_{1},\ldots,p_{d}\} }\right\} }\sqrt{\log(1+4d)}.
\]

\begin{theorem}
\label{KSZ_gen} Let $d,n_{1}, \ldots, n_{d} \geq1$ be positive integers and $%
p_{1},\ldots,p_{d} \in\left[1, \infty\right]$. Then there exist signs $%
\varepsilon_{\mathbf{j}} = \pm1$ and a $d$-linear map $A:
\ell_{p_{1}}^{n_{1}} \times\cdots\times \ell_{p_{d}}^{n_{d}} \to\mathbb{K}$
of the form 
\begin{align*}
A\left( z^{1},\ldots,z^{d}\right) =\sum_{j_{1}=1}^{n_{1}} \cdots
\sum_{j_{d}=1}^{n_{d}} \varepsilon_{\mathbf{j}} z^{1}_{j_{1}}\cdots
z^{d}_{j_{d}},
\end{align*}
such that 
\begin{equation*}
\|A\| \leq(C_{d})^{2\left( 1-\frac{1}{\gamma}\right) } \cdot \left(
\sum_{k=1}^{d}n_{k}\right) ^{1-\frac{1}{\gamma}} \cdot\prod_{k =1}^{d}
n_{k}^{\max\left\{ \frac{1}{\gamma} - \frac{1}{p_{k}},0\right\} }, 
\end{equation*}
with 
$\gamma:= \min\left\{ 2,\max\{p_{k} : p_{k} \leq2\}\right\} $.
\end{theorem}

\begin{proof}
If $p_{k} \geq2$, for all $k=1, \cdots, d$, the estimate above coincide with Proposition \ref{KSZ}. For the sake of clarity, we shall assume that $\{k\in\{1,\cdots,d\}: p_{k} \geq2\} = \{1,\dots,m\}$ and, for $p>0$, define 
\begin{equation*}
\beta(p):= 
\begin{cases}
1-\frac{1}{p}, \ \text{if \,} p \geq2; \\ 
0, \ \text{otherwise}.%
\end{cases}
\end{equation*}

The general form of the Kahane--Salem--Zygmund inequality guarantees the
existence of a $d$-linear map $A: \ell_{p_{1}}^{n_{1}} \times\cdots\times
\ell_{p_{m}}^{n_{m}} \times\ell_{2}^{n_{m+1}} \times\cdots\times\ell
_{2}^{n_{d}} \to\mathbb{K}$ such that 
\begin{align*}
\|A\|_{\mathcal{L}\left( \ell_{p_{1}}^{n_{1}}
\times\cdots\times\ell_{p_{m}}^{n_{m}} \times\ell_{2}^{n_{m+1}}
\times\cdots\times\ell_{2}^{n_{d}}; \mathbb{K}\right) } &\leq C_{d}
\cdot\left( \sum_{k=1}^{d}n_{k}\right)^{\frac{1}{2}} \cdot\prod_{k =1}^{m}
n_{k}^{\frac12 - \frac1{p_{k}}} \\
&= C_{d} \cdot\left( \sum_{k=1}^{d}n_{k}\right) ^{\frac{1}{2}} \cdot
\prod_{k =1}^{d} n_{k}^{\max\left\{ \frac12 - \frac1{p_{k}},0\right\} }.
\end{align*}
The restriction of this map to $\ell_{p_{1}}^{n_{1}} \times\cdots\times
\ell_{p_{m}}^{n_{m}} \times\ell_{1}^{n_{m+1}}
\times\cdots\times\ell_{1}^{n_{d}}$ has norm no greater than 
\begin{equation*}
\prod_{k=1}^{m} n_{k}^{1/p_{k}^{\prime}} = \prod_{k=1}^{d}
n_{k}^{\beta(p_{k})}. 
\end{equation*}

Since $1 \leq\gamma\leq2$ there exists $\theta\in(0,1)$ such that $\frac{1}{%
\gamma} = \frac{1- \theta}{1} + \frac{\theta}{2}$. Note that $\theta=\frac{2%
}{\gamma^{\prime}}$. Now write $\frac{1}{p_{k}} = \frac{1- \theta}{p_{k}} + 
\frac{\theta}{p_{k}}$ for all $k \in\{1, \cdots, m\}$. By using the
Riesz--Thorin interpolation theorem (see \cite[p.18]{BergLofs} or \cite%
{Bayart}), 
\begin{align*}
& \|A\|_{\mathcal{L}\left( \ell_{p_{1}}^{n_{1}} \times\cdots\times
\ell_{p_{m}}^{n_{m}} \times\ell_{\gamma}^{n_{m+1}} \times\cdots\times
\ell_{\gamma}^{n_{d}}; \mathbb{K}\right) } \\
& \leq\|A\|_{\mathcal{L}\left( \ell_{p_{1}}^{n_{1}} \times\cdots\times
\ell_{p_{m}}^{n_{m}} \times\ell_{1}^{n_{m+1}} \times\cdots\times\ell
_{1}^{n_{d}}; \mathbb{K}\right) }^{1- \theta} \cdot\|A\|_{\mathcal{L}\left(
\ell_{p_{1}}^{n_{1}} \times\cdots\times\ell_{p_{m}}^{n_{m}}
\times\ell_{2}^{n_{m+1}} \times\cdots\times\ell_{2}^{n_{d}}; \mathbb{K}%
\right)}^{\theta} \\
& \leq\prod_{k=1}^{d} n_{k}^{(1- \theta)\beta(p_{k})} \left[ C_{d}
\cdot\left( \sum_{k=1}^{d}n_{k}\right) ^{\frac{1}{2}} \cdot \prod_{k =1}^{d}
n_{k}^{\max\left\{ \frac12 - \frac1{p_{k}}, 0\right\} }\right] ^{\theta} \\
& = (C_{d})^{\theta} \cdot\left( \sum_{k=1}^{d}n_{k}\right)^{\frac{\theta}{2}%
} \cdot\prod_{k =1}^{d} n_{k}^{(1- \theta)\beta(p_{k})+\max\left\{\frac{%
\theta}{\gamma} - \frac{\theta}{p_{k}}, 0\right\} } \\
& = (C_{d})^{\frac{2}{\gamma^{\prime}}} \cdot\left( \sum_{k=1}^{d}
n_{k}\right)^{\frac{1}{\gamma^{\prime}}} \cdot\prod_{k =1}^{d}
n_{k}^{\max\left\{ \frac1\gamma- \frac1{p_{k}}, 0\right\} }.
\end{align*}

For each $k \notin\{1, \cdots, m\}$ the unit ball of $\ell_{p_{k}}^{n_{k}}$
lies in the unit ball of $\ell_{\gamma}^{n_{k}}$, so

\begin{align*}
\|A\|_{\mathcal{L}\left( \ell_{p_{1}}^{n_{1}}
\times\cdots\times\ell_{p_{d}}^{n_{d}}; \mathbb{K}\right) } & \leq\|A\|_{%
\mathcal{L} \left( \ell_{p_{1}}^{n_{1}}
\times\cdots\times\ell_{p_{m}}^{n_{m}} \times\ell_{\gamma}^{n_{m+1}}
\times\cdots\times\ell_{\gamma}^{n_{d}}; \mathbb{K}\right) } \\
& \leq(C_{d})^{\frac{2}{\gamma^{\prime}}} \cdot\left( \sum_{k=1}^{d} n_{k}
\right)^{ \frac{1}{\gamma^{\prime}}} \cdot \prod_{k=1}^{d}
n_{k}^{\max\left\{ \frac1\gamma- \frac1{p_{k}}, 0\right\}}.
\end{align*}
\end{proof}

In order to clarify the result above, we illustrate the simpler case, when
dealing with $n_{1} =\dots= n_{d} = n$.

\begin{corollary}
The norm of the $d$-linear form $A$ in Theorem \ref{KSZ_gen}, with $%
n_{1}=\dots=n_{d}=n$, has the following estimate: 
\begin{equation*}
\Vert A\Vert\leq C\cdot n^{1-\frac{1}{\gamma}+\sum_{k=1}^{d}\max\left\{ 
\frac{1}{\gamma}-\frac{1}{p_{k}},0\right\} }, 
\end{equation*}
with $C=\left( d\cdot C_{d}^{2}\right) ^{1-\frac{1}{\gamma}}$ and $%
\gamma:=\min\left\{ 2,\max\{p_{k}:p_{k}\leq2\}\right\} $.
\end{corollary}

The polynomial and vector valued version of the previous result read as
follows.

\begin{corollary}
Let $p\in\left[ 1,\infty\right] $, and integers $n,d\geq1$. Then there exist
signs $\varepsilon_{\mathbf{j}}=\pm1$ and a homogeneous polynomial $P$ of
degree $d$ in the variable $z\in\ell_{p}^{n}$ of the form 
\begin{equation*}
\sum_{\left\vert \alpha\right\vert =d} \varepsilon_{\mathbf{j}} \binom{d}{%
\alpha}z^{\alpha} 
\end{equation*}
such that 
\begin{equation*}
\Vert P\Vert\leq(C_{d})^{2\left( 1-\frac{1}{\gamma}\right) }\cdot n^{1-\frac{%
1}{\gamma}+d\cdot\max\left\{ \frac{1}{2}-\frac{1}{p},0\right\} }, 
\end{equation*}
with $\gamma:=\min\{p,2\}$.
\end{corollary}

\begin{corollary}
Let $d,n_1, \ldots, n_{d+1} \geq 1, \, s,p_1,\ldots,p_d \in \left[1,\infty%
\right]$. Then there exist signs $\varepsilon_{\mathbf{j}} = \pm 1$ and a $d$%
-linear map $A: \ell_{p_1}^{n_1} \times \cdots \times \ell_{p_d}^{n_d}
\to\ell_{s}^{n_{d+1}}$ of the form 
\begin{equation*}
A\left(z^1,\ldots,z^d\right) =\sum_{j_1=1}^{n_1} \ldots \sum_{j_d=1}^{n_d}
\sum_{j_{d+1}=1}^{n_{d+1}} \varepsilon_{\mathbf{j}} z^1_{j_1}\cdots
z^d_{j_d}e_{j_{d+1}},
\end{equation*}
such that 
\begin{equation*}
\|A\| \leq (C_d)^{2 \left(1 - \frac{1}{\gamma} \right)} \cdot
\left(\sum_{k=1}^{d+1}n_k\right)^{1 - \frac{1}{\gamma}} \cdot
\prod_{k=1}^{d+1} n_k^{\max\left\{\frac1\gamma - \frac1{p_k},0\right\}},
\end{equation*}
with $p_{d+1}:=s^{\prime}$ and $\gamma:= \min\left\{2,\max\{p_k : p_k \leq
2\}\right\}$.
\end{corollary}

In order to prove the last result, one just need to use duality and the isometric correspondence between $(d+1)$-linear forms $\ell_{p_1}^{n_1} \times \cdots \times \ell_{p_d}^{n_d}\times \ell_{s^{\prime}}^{n_{d+1}} \to \mathbb{K}$ and $d$-linear maps $\ell_{p_1}^{n_1} \times \cdots \times \ell_{p_d}^{n_d} \to \ell_{s}^{n_{d+1}}$ given by
\begin{align*}
\left(z \mapsto
 \sum_{j_1=1}^{n_1} \ldots \sum_{j_d=1}^{n_d} \sum_{j_{d+1}=1}^{n_{d+1}}
  a_{i_1,\ldots,i_d}^{i_{d+1}} z^1_{j_1} \cdots z^d_{j_d}z^{d+1}_{j_{d+1}}
\right) \\
\mapsto \left(
z \mapsto \sum_{j_1=1}^{n_1} \ldots \sum_{j_d=1}^{n_d} \sum_{j_{d+1}=1}^{n_{d+1}}
 a_{i_1,\ldots,i_d}^{i_{d+1}} z^1_{j_1}\cdots z^d_{j_d}e_{j_{d+1}}
\right).
\end{align*}

\section{Optimality} \label{sec-ksz-opt}

The following result shows that the Kahane--Salem--Zygmund inequality obtained in Theorem \ref{KSZ_gen} is sharp when $p_{1},\dots,p_{d}\in\lbrack2,\infty]$.

\begin{theorem}\label{op_exp}
Let $p_{1},\dots,p_{d}\in\lbrack2,\infty]$. Then 
\begin{align}  \label{est.inf.}
\frac{1}{d \cdot 2^{\frac{d-1}{2}}} \leq\inf \frac{\|A\|} {\left(
n_{1}^{1/2}+\cdots+n_{d}^{1/2}\right) \prod_{j=1}^{d}n_{j}^{\frac{1}{2}-%
\frac{1}{p_{j}}}} \leq C_{d},
\end{align}
with $C_{d}:=8(d!)^{1-\max\left\{ \frac{1}{2},\frac{1}{\max\{p_{1},%
\ldots,p_{d}\}}\right\} }\sqrt{\log(1+4d)}$, for all unimodular $d$-linear
forms $A:\ell_{p_{1}}^{n_1} \times\cdots\times \ell_{p_{d}}^{n_d} \to\mathbb{K}$.
\end{theorem}

\begin{proof} Littlewood's $\left( \ell_{1},\ell_{2}\right)$-inequalities (see \cite{ABPS_HL,ANNPR_POS}) tell us that 
\begin{align*}
\sum\limits_{i_{1}=1}^{n_{1}}\left(
\sum\limits_{i_{2},...,i_{d}=1}^{n_{2},...,n_{d}}\left\vert T\left(
e_{i_{1},...,}e_{i_{d}}\right) \right\vert ^{2}\right) ^{1/2} & \leq\left( 
\sqrt{2}\right) ^{d-1}\left\Vert T\right\Vert , \\
& \vdots \\
\left( \sum\limits_{i_{1},...,i_{d-1}=1}^{n_{1},...,n_{d-1}}\left(
\sum\limits_{i_{d}=1}^{n_{d}}\left\vert T\left(
e_{i_{1},...,}e_{i_{d}}\right) \right\vert \right) ^{2}\right) ^{1/2} &
\leq\left( \sqrt {2}\right) ^{d-1}\left\Vert T\right\Vert ,
\end{align*}
for all $T:\ell_{\infty}^{n_{1}}\times\cdots\times\ell_{\infty}^{n_{d}}%
\rightarrow\mathbb{K}$. Thus%
\begin{align*}
n_{1}n_{2}^{1/2}...n_{d}^{1/2} & \leq\left( \sqrt{2}\right) ^{d-1}\left\Vert
T\right\Vert , \\
& \vdots \\
n_{1}^{1/2}...n_{d-1}^{1/2}n_{d} & \leq\left( \sqrt{2}\right)
^{d-1}\left\Vert T\right\Vert ,
\end{align*}
for all unimodular $T:\ell_{\infty}^{n_{1}}\times\cdots\times\ell_{\infty
}^{n_{d}}\rightarrow\mathbb{K}$. Then
\begin{equation*}
\left\Vert T\right\Vert \geq
\frac{1}{d\left( \sqrt{2}\right) ^{d-1}}\left( n_{1}^{1/2}+\cdots+n_{d}^{1/2}\right) \prod\limits_{j=1}^{d}n_{j}^{\frac {1}{2}},
\end{equation*}
for all unimodular $T:\ell_{\infty}^{n_{1}}\times\cdots\times\ell_{\infty
}^{n_{d}}\rightarrow\mathbb{K}$. Now consider a unimodular $%
S:\ell_{p_{1}}^{n_{1}}\times\cdots\times\ell_{p_{d}}^{n_{d}}\rightarrow%
\mathbb{K}$. Let us denote by $S_{\infty}$ the same operator but now with
domain $\ell_{\infty }^{n_{1}}\times\cdots\times\ell_{\infty}^{n_{d}}.$ We
have%
\begin{equation*}
\left\Vert S_{\infty}\right\Vert \geq\frac{1}{d\left( \sqrt{2}\right) ^{d-1}}%
\left( n_{1}^{1/2}+\cdots+n_{d}^{1/2}\right) \prod\limits_{j=1}^{d}n_{j}^{%
\frac{1}{2}}
\end{equation*}
and hence there exist $x^{(1)}, \dots, x^{(d)}$ with $\left\vert x_{j}^{(k)}\right\vert =1$ for all $k,j$ such that
\begin{equation*}
\left\vert S_{\infty}\left( x^{(1)},...,x^{(d)}\right) \right\vert \geq 
\frac{1}{d\left( \sqrt{2}\right) ^{d-1}}\left( n_{1}^{1/2}+\cdots
+n_{d}^{1/2}\right) \prod\limits_{j=1}^{d}n_{j}^{\frac{1}{2}}. 
\end{equation*}
Consider%
\begin{equation*}
y^{(k)}=\frac{1}{n_{k}^{1/p_{k}}}x^{(k)}
\end{equation*}
for all $k=1,\dots,d$. Since each $y^{(k)}$ belongs to the closed unit ball of $\ell_{p_{k}}^{n_{k}}$ we have 
\begin{align*}
\left\Vert S\right\Vert & \geq\left\vert S\left( y^{(1)},...,y^{(d)}\right)
\right\vert \\
& =\frac{1}{n_{1}^{1/p_{1}}\cdots n_{d}^{1/p_{d}}}\left\vert S\left(
x^{(1)},...,x^{(d)}\right) \right\vert \\
& =\frac{1}{n_{1}^{1/p_{1}}\cdots n_{d}^{1/p_{d}}}\left\vert S_{\infty
}\left( x^{(1)},...,x^{(d)}\right) \right\vert \\
& \geq\frac{1}{n_{1}^{1/p_{1}}\cdots n_{d}^{1/p_{d}}}\frac{1}{d\left( \sqrt{2%
}\right) ^{d-1}}\left( n_{1}^{1/2}+\cdots+n_{d}^{1/2}\right)
\prod\limits_{j=1}^{d}n_{j}^{\frac{1}{2}} \\
& =\frac{1}{d\left( \sqrt{2}\right) ^{d-1}}\left( n_{1}^{1/2}+\cdots
+n_{d}^{1/2}\right) \prod\limits_{j=1}^{d}n_{j}^{\frac{1}{2}-\frac{1}{p_{j}}%
}.
\end{align*}
We thus get that 
\begin{equation*}
\frac{1}{d\left( \sqrt{2}\right) ^{d-1}}\leq\inf\frac{\left\Vert
A\right\Vert }{\left( n_{1}^{1/2}+\cdots+n_{d}^{1/2}\right) \prod
\limits_{j=1}^{d}n_{j}^{\frac{1}{2}-\frac{1}{p_{j}}}}. 
\end{equation*}
On the other hand, by Theorem \ref{KSZ_gen},\ for all $p_{1},\dots,p_{d}\in\lbrack2,\infty]$,
there is a constant $C_{d}>0$ and a multilinear form $T_{0}:%
\ell_{p_{1}}^{n_{1}} \times\cdots\times\ell_{p_{d}}^{n_{d}} \to \mathbb{K}$
such that 
\begin{equation*}
\left\Vert T_{0}\right\Vert \leq C_{d}\left( n_{1}^{1/2}+\cdots+n_{d}^{1/2}
\right) \prod_{j=1}^{d}n_{j}^{\frac{1}{2}-\frac{1}{p_{j}}}. 
\end{equation*}
Therefore, we conclude that 
\[
\frac{1}{d\left( \sqrt{2}\right)^{d-1}} \leq\inf\frac{\|A\|}{\left(n_{1}^{1/2}+\cdots+n_{d}^{1/2}\right) \prod \limits_{j=1}^{d}n_{j}^{\frac{1}{2}-\frac{1}{p_{j}}}}
\leq C_{d}.
\]
\end{proof}

\begin{remark}[Case $1\leq p \leq 2$]
The search for the optimal constants and exponents involved are natural and relevant topics for further investigation in this framework. The determination of the unknown optimal exponents for the case $p \in (1,2)$ of Theorem \ref{KSZ_gen} rely in an open result on the interpolation of certain multilinear forms (see \cite{AP}): it is well-known that every $d$-linear form on $\ell_p^{n} \times\cdots\times \ell_p^{n}$ is multiple $(r,1)$-summing, with 
\begin{equation*}
r = \frac{dp}{d+p-1}, \quad \text{ for } p=1 \text{ or } p=2. 
\end{equation*}

What about intermediate results for $p \in (1,2)$? If it was possible an interpolative approach would tell us that every $d$-linear operators from $\ell_p^{n} \times\cdots\times \ell_p^{n}$ is multiple $(dp/(d+p-1),1)$-summing, and a simple computation as in the proof of Theorem \ref{op_exp} would tell us that the exponents in Theorem \ref{KSZ_gen}, with $p_1=\cdots=p_d=p \in (1,2)$ and $n_1=\cdots=n_d=n$, would be in fact admissible exponents and give us the estimate for asymptotic behavior pursued in Problem \ref{prob} (see also \eqref{1s}): 
\begin{equation*}
s \geq 1-\frac{1}{p}. 
\end{equation*}
Even in the linear case, similar problems remain open. Based on this discussion presented, we conjecture the following optimal result that also would complete the solution of Problem \ref{prob}.
\end{remark}

\begin{conjecture}
Let $p_{1},\dots,p_{d}\in [1,\infty]$. There exist $B_{d},C_d>0$ such that 
\begin{equation*}
B_d \leq \inf \frac{\|A\|}{ \left(\sum_{k=1}^{d} n_{k}^{1-\frac{1}{\gamma}%
}\right) \cdot \prod_{k =1}^{d} n_{k}^{\max \left\{ \frac{1}{\gamma} - \frac{%
1}{p_{k}},0\right\}} } \leq C_d, 
\end{equation*}
with 
$\gamma:= \min\left\{ 2,\max\{p_{k} : p_{k} \leq2\}\right\}$, for all unimodular $m$-linear forms $A:\ell_{p_{1}}^{n_1} \times \cdots \times \ell_{p_{d}}^{n_d} \to\mathbb{K}$ and the exponents involved are sharp.
\end{conjecture}

\bigskip

Notice that Theorem \ref{op_exp} provides the following asymptotic growth of 
$\inf \|A\|$ with respect to $n_1,\dots,n_d$: 
\begin{equation*}
\frac{1}{d \cdot 2^{\frac{d-1}{2}}} \leq\inf \frac{\|A\|} {%
\left(\sum_{j=1}^d n_{j}^{1/2}\right) \prod_{j=1}^{d}n_{j}^{\frac{1}{2}-%
\frac{1}{p_{j}}}} \leq 8(d!)^{1-\max\left\{ \frac{1}{2},\frac{1}{p}\right\} }%
\sqrt{\log(1+4d)}, 
\end{equation*}
where $p:=\max\{p_{1},\ldots,p_{d}\}$. Clearly there is a far distance
between the lower and upper estimates. Some interesting problems for further
investigation are presented next.

\begin{problem}
Is there a better lower bound to replace 
\begin{equation*}
d^{-1}2^{\frac{1-d}{2}} ? 
\end{equation*}
\end{problem}

\begin{problem}
Is there a better upper bound to replace 
\begin{equation*}
8(d!)^{1-\max\left\{ \frac{1}{2}, \frac{1}{\max\{p_{1},\ldots,p_{d}\}}\right\} }\sqrt{\log(1+4d)} ? 
\end{equation*}
\end{problem}

\begin{problem}
In the complex case, what about unimodular forms with complex coefficients with modulo $1$, instead of just $1$ and $-1$? It is obvious that \eqref{est.inf.} holds. What about the lower and upper bounds?
\end{problem}

\section{Applications} \label{sec-app}

The Kahane--Salem--Zygmund inequality is often used to prove the optimality of exponents in the Hardy--Littlewood inequality. The next results follow in this vein. The fact that the multilinear form provided in Theorem \ref{KSZ_gen} is defined on $\ell_{p_1}^{n_1}\times \cdots \times \ell_{p_d}^{n_d}$ for arbitrary finite dimensions $n_1, \cdots, n_d$ has a crucial role. Throughout this section, $X_{p}$ stands for $\ell_{p}$ if $1\leq p<\infty$ and $X_{\infty}:=c_{0}$.

\subsection{On a result of Santos and Velanga}

From now on, the symbol $e_{j}^{n_{j}}$ stands for $(e_{j},\overset{n_{j}times}{\dots },e_{j})$, with $e_{j}$ the $j$-th canonical vector. The following result, due to Santos and Velanga \cite[Theorem 1.2]{SV}, is a complete classification of the multilinear Bohnenblust--Hille type inequalities.

\begin{theorem}
Let $1 \leq k \leq d$ and $m_1,\dots,m_k$ be positive integers such that $%
m_1+\cdots+m_k = d$. Also let ${\bm{\rho}} := \left(\rho_1, \ldots,
\rho_k\right) \in \left(0,\infty\right)^{k}$. The following assertions are
equivalent:

\begin{itemize}
\item[(I)] There is a constant $B_{k,\bm{\rho}}^{\mathbb{K}} \geq 1$ such
that, for all $d$-linear forms $T: c_0 \times \cdots \times c_0 \to \mathbb{K%
}$, 
\begin{equation*}
\left(\sum_{j_1=1}^{\infty} \left( \cdots \left( \sum_{j_k=1}^{\infty}
\left|T\left(e_{j_1}^{m_1},\dots, e_{j_k}^{m_k}\right)\right|^{\rho_k}
\right)^{\frac{\rho_{k-1}}{\rho_k}}\cdots \right)^{\frac{\rho_1}{\rho_2}}
\right)^{\frac{1}{\rho_1}} \leq B_{k,\bm{\rho}}^{\mathbb{K}} \|T\|.
\end{equation*}

\item[(II)] For all $\mathcal{I }\subset \{1,\dots,k\}$, $\sum_{j\in 
\mathcal{I}} \frac{1}{\rho_j} \leq \frac{\text{card\,}\mathcal{I}+1}{2}. $
\end{itemize}
\end{theorem}

The main result of this section shows that (I)$\Rightarrow $(II) is valid in
the general setting of the Hardy--Littlewood inequality, which deals with
both $c_0$ or $\ell_p$ spaces. We show that Theorem \ref{HL_genblocks}
provides a straightforward alternative proof of (I)$\Rightarrow $(II) of
Santos and Velanga.

\begin{theorem}
\label{HL_genblocks} Let $1 \leq k \leq d$ and $m_1,\dots,m_k$ be positive
integers such that $m_1+\cdots+m_k = d$. Also let 
\begin{equation*}
\mathbf{p}:= \left( \mathbf{p}^{1},\dots,\mathbf{p}^{k} \right) \in [1,\, +
\infty]^{d}, \quad \mathbf{p}^{{j}}:=\left(p_1^{j}, \dots,p_{m_j}^{j}\right)
\in (1, \, + \infty]^{m_j},
\end{equation*}
with $j=1,\dots,k$ and $\mathbf{\rho} := \left(\rho_1, \ldots, \rho_k\right)
\in \left(0, + \infty\right)^{k}$. If there is a constant $C_{k,\bm{\rho}, 
\mathbf{p}}^{\mathbb{K}} \geq 1$ such that 
\begin{align}
\left(\sum_{j_1=1}^{\infty} \left( \cdots \left( \sum_{j_k=1}^{\infty}
\left|T\left(e_{j_1}^{m_1},\dots, e_{j_k}^{m_k}\right)\right|^{\rho_k}
\right)^{\frac{\rho_{k-1}}{\rho_k}}\cdots \right)^{\frac{\rho_1}{\rho_2}}
\right)^{\frac{1}{\rho_1}} \leq C_{k,\bm{\rho}, \mathbf{p}}^{\mathbb{K}}
\|T\|,  \label{HL_desgenblocks}
\end{align}
for all $d$-linear forms $T: X_{p_1^{1}} \times \cdots \times
X_{p_{m_1}^{1}} \times \cdots \times X_{p_{1}^{k}} \times \cdots
X_{p_{m_k}^{k}} \to \mathbb{K}$. Then 
\begin{equation*}
\sum_{j\in \mathcal{I}} \frac{1}{\rho_j} \leq \frac{\text{card\,}\mathcal{I}%
+1}{2} - \sum_{j \in \mathcal{I}} \left|\frac{1}{\mathbf{p}^{j}}\right|, 
\end{equation*}
for all $\mathcal{I }\subset \{1,\dots,k\}$.
\end{theorem}

\begin{proof}
Let us suppose \eqref{HL_desgenblocks} holds. Since $\mathcal{L}\left(%
\widehat{\otimes}_{1\leq j \leq m_1}^{\pi} X_{p_j^{1}}, \cdots, \widehat{%
\otimes}_{1\leq j \leq m_k}^{\pi} X_{p_j^{k}}; \mathbb{K}\right)$ and $%
\mathcal{L}\left(X_1, \cdots, X_d; \mathbb{K}\right)$ are isometric, we get
an operator 
\begin{equation*}
\widehat{T}: \widehat{\otimes}_{1\leq j \leq m_1}^{\pi} X_{p_j^{1}}\times
\cdots \times \widehat{\otimes}_{1\leq j \leq m_k}^{\pi} X_{p_j^{k}}
\rightarrow\mathbb{K}, 
\end{equation*}
that satisfies 
\begin{equation*}
\widehat{T} \left( \otimes_{1\leq j \leq m_1} x_{j}^{1}, \cdots,
\otimes_{1\leq j \leq m_k} x_{j}^{k}\right) = T\left(x_1^{1},
\cdots,x_{m_1}^{1}, \cdots, x_1^{k}, \cdots,x_{m_k}^{k} \right). 
\end{equation*}
Then, 
\begin{equation*}
\widehat{T} \left( \otimes_{m_1} e_{j_1}, \cdots, \otimes_{m_k}
e_{j_{k}}\right) = T\left(e_{j_1}^{m_1}, \cdots,e_{j_k}^{m_k} \right) 
\end{equation*}
and $\|\widehat{T}\|=\|T\|$. Using \cite[Theorem 1.3]{AF} it is possible to
prove that it is sufficient to deal with $k$-linear forms $A: X_{r_1} \times
\cdots \times X_{r_k} \to \mathbb{K}$, with $\frac{1}{r_j} = \left|\frac{1}{%
\mathbf{p}^{j}}\right|,\, j=1,\dots,k$, defined by 
\begin{equation*}
A\left(y^1, \cdots, y^k\right):= \widehat{T}\left(u_1(y^1), \cdots,
u_k(y^k)\right), 
\end{equation*}
which is bounded and fulfills $\|A\| \leq \|\widehat{T}\|$. Therefore, 
\begin{equation*}
\left( \sum_{j_1=1}^{n_1} \left( \cdots \left( \sum_{j_d=1}^{n_k} \left|
A(e_{j_1},\dots,e_{j_k}) \right|^{\rho_k} \right)^\frac{\rho_{k-1}}{\rho_k}
\cdots \right)^\frac{\rho_1}{\rho_2} \right)^\frac1{\rho_1} \leq C_{k,%
\bm{\rho}, \mathbf{p}}^{\mathbb{K}} \|A\|,
\end{equation*}
for all positive integers $n_1,\dots,n_k$ and all $k$-linear forms $A:
X_{r_1} \times \cdots \times X_{r_k} \to \mathbb{K}$, with $\frac{1}{r_j} =
\left|\frac{1}{\mathbf{p}^{j}}\right|,\, j=1,\dots,k. $ Using the map from
the Theorem \ref{KSZ_gen}, we conclude that, for all positive integers $%
n_1,\dots,n_k$, 
\begin{equation*}
n_1^{\frac1{\rho_1}} \cdots n_k^{\frac1{\rho_k}} \leq C_d \cdot
n_1^{\frac12- \frac1{r_1}} \cdots n_k^{\frac12 - \frac1{r_k}} \cdot
\left(n_1 + \cdots + n_k \right)^\frac12.
\end{equation*}
Therefore,
\begin{equation*}
\sum_{j \in \mathcal{I}}\frac1{\rho_j} \leq \frac{\text{card\,}\mathcal{I}+1%
}{2} - \sum_{j \in \mathcal{I}} \left|\frac{1}{\mathbf{p}^{j}}\right|.
\end{equation*}
\end{proof}

\subsection{Dimension blow-up estimates}

We turn our attention to the order in which the estimates blow-up with respect to the dimension when the exponents $(\rho _{1},\dots,\rho _{d})$ do not fulfill the range required in Hardy--Littlewood inequality (see \cite{ABPS_HL}). In this framework, there is no constant $C_{d,\bm{\rho}, \mathbf{p}}^{\mathbb{K}} \geq 1$, independent of the dimensions $n_1,\dots,n_d$, for which the inequality 
\begin{align*}
\left(\sum_{j_1=1}^{n_1} \left( \cdots \left( \sum_{j_d=1}^{n_d}
\left|T\left(e_{j_1},\dots, e_{j_d}\right)\right|^{\rho_d} \right)^{\frac{%
\rho_{d-1}}{\rho_d}}\cdots \right)^{\frac{\rho_1}{\rho_2}} \right)^{\frac{1}{%
\rho_1}} \leq C_{d,\bm{\rho}, \mathbf{p}}^{\mathbb{K}} \|T\|,
\end{align*}
is universally complied for all $d$-linear forms $T: X_{p_1} \times \cdots \times X_{p_{d}} \to \mathbb{K}$. Recent papers present results on this line \cite{ANNPR_POS,PSST}. The next one follows this fashion.

\begin{theorem}{\cite[Theorem 5]{ANNPR_POS}}
Let $d\geq 2$ be an integer. If $\left|\frac{1}{\mathbf{p}} \right|\leq \frac{1}{2}$ and $\mathbf{r} \in [1,2]^d$, then there is a constant $C_{d,\mathbf{r},\mathbf{p}}^{\mathbb{K}}\geq 1$ (not depending on $n$) such that 
\begin{multline*}
\left( \sum_{i_{1}=1}^{n} \left( \cdots \left( \sum_{i_{d}=1}^{n} \left\vert
T(e_{i_{1}}, \dots,e_{i_{d}}) \right\vert^{r_{d}} \right)^{\frac{r_{d-1}}{%
r_{d}}}\cdots \right)^{\frac{r_{1}}{r_{2}}} \right)^{\frac{1}{r_{1}}} \\
\leq C_{d, \mathbf{r}, \mathbf{p}}^{\mathbb{K}} \cdot n^{\max \left\{ \left| 
\frac{1}{\mathbf{r}} \right| -\frac{d+1}{2} + \left| \frac{1}{\mathbf{p}}
\right|, 0\right\}} \Vert T\Vert,
\end{multline*}
for all $d$-linear forms $T:\ell _{p_1}^{n}\times \cdots \times
\ell_{p_d}^{n} \to \mathbb{K}$ and all positive integer $n$. Moreover, the
exponent $\max \left\{ \left| \frac{1}{\mathbf{r}} \right| -\frac{d+1}{2}%
+\left| \frac{1}{\mathbf{p}} \right|, 0\right\} $ is optimal.
\end{theorem}

Our goal is to improve this result, showing the dependence arisen in distinct dimensions $n_1,\dots,n_d$. The result obtained reads as follows.

\begin{proposition}
Let $d \geq 2$ be an integer, $\mathbf{p}:= (p_1, \ldots, p_d) \in [1, \, +
\infty]^{d}$ be such that $\left|\frac{1}{\mathbf{p}}\right| \leq \frac{1}{2}
$ and also let $s_{k}, \rho_k \in \left(0,\infty\right)$, for $k=1,\dots,d$.
Suppose that there exists $D_{d,\bm{\rho}, \mathbf{p}}^{\mathbb{K}} \geq 1$
such that 
\begin{equation*}
\left(\sum_{j_1=1}^{n_1} \left( \ldots \left(\sum_{j_d=1}^{n_d}
\left|T\left(e_{j_1},\dots, e_{j_d}\right)\right|^{\rho_d} \right)^{\frac{%
\rho_{d-1}}{\rho_d}}\dots \right)^{\frac{\rho_1}{\rho_2}} \right)^{\frac{1}{%
\rho_1}} \leq D_{d,\bm{\rho}, \mathbf{p}}^{\mathbb{K}} n_1^{s_1} \cdots
n_d^{s_d} \|T\|,
\end{equation*}
for all bounded $d$-linear operators $T: \ell_{p_1}^{n_1} \times \cdots
\times \ell_{p_d}^{n_d} \to \mathbb{K}$ and any positive integers $%
n_1,\ldots, n_d$. Then for all $\mathcal{I }\subset \{1,\dots,d\}$, 
\begin{equation*}
\sum_{j\in \mathcal{I}} s_j \geq \max \left\{0, \sum_{j\in \mathcal{I}} 
\frac{1}{\rho_j} - \frac{\text{card\,}\mathcal{I}+1}{2}+\sum_{j\in \mathcal{I%
}} \frac{1}{p_j}\right\}.
\end{equation*}
\end{proposition}

\begin{proof}
Let us consider the $d$-linear form $A: \ell_{p_1}^{n_1} \times \cdots\times
\ell_{p_d}^{n_d} \to \mathbb{K}$ given by Theorem \ref{KSZ_gen}. Since, $%
\left|A\left(e_{j_1}, \ldots, e_{j_d}\right)\right| = 1$ for any $%
j_1,\ldots, j_d$, we have 
\begin{equation*}
\left(\sum_{j_1=1}^{n_1} \left( \ldots \left(\sum_{j_d=1}^{n_d}
\left|A\left(e_{j_1},\dots, e_{j_d}\right)\right|^{\rho_d} \right)^{\frac{%
\rho_{d-1}}{\rho_d}}\dots \right)^{\frac{\rho_1}{\rho_2}} \right)^{\frac{1}{%
\rho_1}} = n_1^{\frac{1}{\rho_1}} \cdots n_d^{\frac{1}{\rho_d}}.
\end{equation*}
By Theorem \ref{KSZ_gen}, 
\begin{align*}
n_1^{\frac{1}{\rho_1}} \cdots n_d^{\frac{1}{\rho_d}} &\leq D_{d,\bm{\rho},%
\mathbf{p}}^{\mathbb{K}} n_1^{s_1} \cdots n_d^{s_d} \|A\| \\
&\leq D_{d,\bm{\rho}, \mathbf{p}}^{\mathbb{K}} n_1^{s_1 + \frac12 -
\frac1{p_1}} \cdots n_d^{s_d + \frac12 - \frac1{p_d}} \left(n_1 + \cdots +
n_d\right)^\frac12
\end{align*}
and so 
\begin{equation*}
\frac{n_1^{\frac{1}{\rho_1} - s_1 + \frac{1}{p_1} - \frac{1}{2}}\cdots n_d^{%
\frac{1}{\rho_d}- s_d + \frac{1}{p_d} - \frac{1}{2}}}{\left(n_1 + \cdots +
n_d\right)^{\frac{1}{2}}} \leq D_{d,\bm{\rho}, \mathbf{p}}^{\mathbb{K}}.
\end{equation*}
Therefore, 
\begin{equation*}
\sum_{j\in \mathcal{I}} \frac{1}{\rho_j} - \sum_{j\in \mathcal{I}} s_{j} - 
\frac{\text{card\,}\mathcal{I}+1}{2}+ \sum_{j\in \mathcal{I}} \frac{1}{p_j}
\leq 0.
\end{equation*}
\end{proof}

\bigskip

\noindent\textbf{Acknowledgments} The authors thank the anonymous referee, whose reading, insightful and important suggestions were crucial to improve and clarify the presentation of the paper and to minimize imprecisions of the original version.


\begin{thebibliography}{99}
\bibitem{ABPS_HL} N. Albuquerque, F. Bayart, D. Pellegrino, J. B. Seoane-Sep\'ulveda, \emph{Optimal Hardy--Littlewood type inequalities for polynomials and multilinear operators}, Israel J. Math., \textbf{211} (2016), 197--220.

\bibitem{ANNPR_POS} N. Albuquerque, T. Nogueira, D. N{\'u}{\~n}ez-Alarc\'on, D. Pellegrino, P. Rueda, \emph{Some applications of the H\"older inequality for mixed sums}, Positivity \textbf{21} (2017), 1575--1592.

\bibitem{AP} G. Ara\'ujo, D. Pellegrino, \emph{A Gale--Berlekamp permutation--switching problem in higher dimensions}, European Journal of Combinatorics \textbf{77} (2019), 17--30.

\bibitem{AF} A. Arias and J. D. Farmer, \textit{On the structure of tensor products of $\ell_p$-spaces.} Pacific J. Math. \textbf{175} (1996): 13--37.

\bibitem{Bauer} H. Bauer, \emph{Measure and Integration Theory}, de Gruyter Studies in Mathematics 26, 2001.

\bibitem{Bayart} F. Bayart, \emph{Maximum modulus of random polynomials}, Quart. J. Math. \textbf{63} (2012) 21--39.

\bibitem{BPS_adv} F. Bayart, D. Pellegrino and J.B. Seoane-Sep\'{u}lveda, \emph{The Bohr radius of the $n$-dimensional polydisk is equivalent to $\sqrt{(\log n)/n}$}, Adv. Math. \textbf{264} (2014), 726--746.

\bibitem{BergLofs} J. Bergh and J. L{\"o}fstr{\"o}m, \emph{Interpolation spaces. An introduction}, Grundlehren der Mathematischen Wissenschaften, No. 223, Springer-Verlag, Berlin, (1976).

\bibitem{Boas2000} H.P. Boas, \emph{Majorant series}, J. Korean Math. Soc., \textbf{37} (2000), 321--337.

\bibitem{BK} H.P. Boas and D. Khavinson, \emph{Bohr's power series theorem in several variables}, Proc. Amer. Math. Soc. \textbf{125} (1997) 2975--2979.

\bibitem{PSST} D. Pellegrino, J. Santos, D. Serrano and E. Teixeira, \emph{A regularity principle in sequence spaces and applications}, Bull. Sci. Math. \textbf{141} (2017) 802--837.

\bibitem{SV} J. Santos and T. Velanga, \emph{On the Bohnenblust--Hille inequality for multilinear forms}, Results Math. \textbf{72} (2017), no. 1-2, 239--244.
\end{thebibliography}
\end{document}